\documentclass{amsart} 
\usepackage{graphicx}
\usepackage{amsmath, amsthm}

\newtheorem{proposition}{Proposition}

\begin{document}

\title[Encoding paths in graphs]{Encoding shortest paths in graphs assuming the code is queried using bit-wise comparison}
\author{G. Caylak Kayaturan and A. Vernitski}
\address{Department of Mathematical Sciences, University of Essex, UK}

\begin{abstract}
One model of message delivery in a computer network is based on labelling each edge by a subset of a (reasonably small) universal set, and then encoding a path as the union of the labels of its edges. Earlier work suggested using random edge labels, and that approach has a disadvantage of producing errors (false positives). We demonstrate that if we make an assumption about the shape of the network (in this paper we consider networks with a dense core and a tree-like periphery) and assume that messages are delivered along shortest paths, we can label edges in a way which prevents any false positives.
\end{abstract}

\maketitle


\section{Introduction\label{sec:Introduction}}

Consider an undirected graph $G=(V,E)$ and a universal set $U$.
The graph $G$ models a computer network, and $U$, as we shall see,
models the header of a message sent from one computer in $G$ to another;
accordingly, we assume that the size $|V|$ and $|E|$ is approximately
in the range $10^{3}-10^{5}$ and the size $|U|$ is approximately
in the range $10^{2}-10^{3}$. Suppose each edge $e\in E$ is labelled
by a subset of $U$; we shall denote the label of $e$ by $[e]\subseteq U$.
The label of a set of edges $S\subseteq E$ is defined as $[S]=\bigcup_{e\in S}[e]$.
We shall say that an edge $e\in E$ is \emph{recognised} by a label
$[S]$ if $[e]\subseteq [S]$. Obviously, if $e\in S$ then $e$ is
recognised by $[S]$; however, it is also possible that $e\not\in S$
and $e$ is recognised by $[S]$; one refers to this situation as
a \emph{false positive}. We shall say that a set of edges $S$ is
\emph{represented faithfully} by its label $[S]$ if none of the edges
$e\not\in S$ which are adjacent to $S$ are recognised by $[S]$.
A particular scenario we have in mind is when $S$ is a path connecting
vertices $u$ and $v$, and $[S]$ is used for routing a message from
$u$ to $v$ (or from $v$ to $u$); thus, $[S]$ is sent along with
the message as its header. We assume that each vertex $v\in V$ is a computer which
cannot access information about the general shape of the network when
it is used for routing messages, but can access the labels of the
edges which are incidental to $v$; accordingly, $v$ can compare
the header of the message $[S]$ with these labels and decide along
which edge the message should be sent next. If $S$ is represented
faithfully by $[S]$ then at each vertex on $S$ it is clear from
inspecting $[S]$ along which edge the message should be sent next.
However, if $S$ is not represented faithfully by $[S]$, that is,
there is a false positive $f\in E$ adjacent to the path $S$ then
it will be impossible to find out from inspecting $[S]$ whether the
message should be sent along $f$ or not. 

In practice, a subset $[S]\subseteq U$ would be represented as a
binary array of length $|U|$, in which each position corresponds to one
fixed element of $U$; in the array representing $[S]$, a bit at
a certain position equal to $0$ (or $1$) means that this element
of $U$ does not belong to $S$ (or belongs to $S$). Thus, the header
attached to the message to describe where and by what route it should
be delivered has size $|U|$. When a computer at a vertex $v$ decides
where to forward the message, it considers each edge $e$ incidental
to $v$ and checks whether $[e]\subseteq[S]$; in practice, this comparison
is implemented as a bitwise comparison of two binary arrays of length
$|U|$ which represent $[e]$ and $[S]$; this operation can be performed
very fast; in fact, it can be performed while the header is passing
through $v$ (for example, as an optical signal), without the need
to store the bits of the header at $v$ and then perform any arithmetic
operations on them. Such fast performance makes this model an attractive
possibility for routing in computer networks \cite{carrea2014optimized}, \cite{Jokela2009}.

One useful labelling (which we shall refer to as the \emph{bit-per-edge
labelling}) is to take $U=E$ and, for each $e\in E$, $[e]=\{e\}$.
Thus, the label of a set $S\subseteq E$ is simply the set of edges
in $S$, that is, $[S]=S$. Note that this labelling represents faithfully
not only every path, but every subset of $E$. The shortcoming of
this labelling is that unless $G$ is small, the size $|U|$, which
is equal to $|E|$, is too large to be usable. 

In our research we assume that $S$ is a path and, more precisely,
if $S$ is a path from a vertex $u$ to a vertex $v$, we assume that
$S$ is one of the shortest paths from $u$ to $v$. Our research concentrates
on looking for ways of labelling edges of a given graph so that, on
the one hand, each shortest path is represented faithfully, and, on
the other hand, the size $|U|$ is reasonably small. In our previous
research we studied the cases when $G$ is a square grid or a hexagonal
grid \cite{caylak2016,kayaturan2016}; without going into much detail, in both cases we found
labellings such that each shortest path is represented faithfully
and the size $|U|$ is of the order $O(\sqrt{|V|})$ or, equivalently,
$O(\sqrt{|E|})$.

In this paper we concentrate on considering graphs that have a
dense core and a tree-like periphery, because some computer networks
have this shape \cite{yook2002modeling, zegura1996model} or are approximated by trees \cite{li2012esm},\cite{reinhardt2012cbfr},\cite{yu2009buffalo},\cite{zhang2010sapper, zhao2010towards} or stars \cite{day1992arrangement, day1994comparative, awwad2003topological}.

We describe how for such a graph, a labelling can be
defined which represents each shortest path faithfully; actually,
our labellings satisfy a stronger property: for each shortest path $S$ and each edge $e$, $[e]\subseteq [S]$ if and only if $e\in S$; that is, our labellings produce no false positives
at all, providing that they are used only for labelling shortest paths and not other sets of edges. 

Our methodology in this paper is as follows: we prove that the labellings we introduce produce no false positives; then we check experimentally what size $|U|$ is required by these labellings, and whether it is within a realistic range.

In practical applications, it is sometimes possible that a message
must be delivered via a path which is not a shortest path. Also, sometimes
a message must be delivered to multiple destinations (this is what
is called multicast, as opposed to unicast). In this paper
we do not consider these generalisations. Another direction of research
is to consider directed graphs, with edges pointing in opposite
directions (that is, from vertex $u$ to vertex $v$ and from vertex
$v$ to vertex $u$) having distinct labels; in this paper we consider
undirected graphs.

\section{Labelling edges\label{sec:Other-ways-of}}

\subsection{A bit per vertex\label{sub:Bit-per-vertex}}

Suppose $G$ is a dense graph; then $|E|$ is relatively large, and
the bit-per-edge labelling from Section \ref{sec:Introduction} is
obviously not optimal. Instead, let us consider a labelling such that
$U=V$ and, for each $e\in E$, $[e]=\{u,v\}$, where $u$ and $v$
are the end vertices of $e$; we shall call it the \emph{bit-per-vertex
labelling}.
\begin{proposition}
If the bit-per-vertex labelling is used to represent a shortest path then it has no false positives.\end{proposition}
\begin{proof}
Indeed, consider a shortest path from a vertex $v_{0}$ to a vertex
$v_{n}$ consisting of edges $e_{k}=\{v_{k},v_{k+1}\}$, where $k=0,\dots,n-1$.
Suppose that there is a false positive $f\in E$; hence, $f=\{v_{i},v_{j}\}$
for some $i,j$; let us assume that $i<j$. Since $f$ is a false
positive, it does not coincide with any $e_{k}$; hence, $j-i>1$.
Consider a path consisting of edges $\{v_{0},v_{1}\}$; $\dots$;
$\{v_{i-1},v_{i}\}$; $f=\{v_{i},v_{j}\}$; $\{v_{j},v_{j+1}\}$;$\dots$;
$\{v_{n-1},v_{n}\}$. This path is a path from $v_{0}$ to $v_{n}$
whose length is less than $n$; this conclusion contradicts the assumption
that the path consisting of edges $e_{k}$ is a shortest path. Thus,
there are no false positives.
\end{proof}

\subsection{Encoding for star graphs\label{sub:Encoding-for-star}}

Consider a star graph $G$ with $n$ edges. For a non-negative integer
$R$, let $K$ be defined as $\left\lceil \sqrt[R]{n}\:\right\rceil $,
that is, the smallest integer which is not less than $\sqrt[R]{n}$.
Thus, $K^{R}\ge n$; therefore, there is a one-to-one mapping from
$E$ to the set of $R$-tuples of integers in the range $\{0,\dots,K-1\}$;
for an edge $e$, let us denote the corresponding tuple by $(\pi_{1}(e),\dots,\pi_{R}(e))$.
Let $U$ consist of pairs $(r,k)$ for each $r=1,\dots,R$ and $k=0,\dots,K-1$
and triples $(r,s,k)$ for all $1\le r<s\le R$ and $k=0,\dots,K-1$.
Let the label $[e]$ of an edge $e$ include pairs $(r,\pi_{r}(e))$
for each $r=1,\dots,R$ and triples $(r,s,\pi_{r}(e)+\pi_{s}(e))$
for all $1\le r<s\le R$, with addition $\pi_{r}(e)+\pi_{s}(e)$ performed
modulo $K$. Let us refer to this labelling as the \emph{star labelling}.
\begin{proposition}
If the star labelling is used to represent a shortest path then it has no false positives.\end{proposition}
\begin{proof}
Indeed, consider a shortest path $P$ in a star graph $G$. If it
consists of $0$ edges or $1$ edge, it is obvious that there can
be no false positives. Now suppose it consists of two edges $e$ and
$f$, and assume there is a false positive $g\in E$. 

Since $g\neq e$, there is $r$ such that $\pi_{r}(g)\neq\pi_{r}(e)$.
For this value of $r$, in $[P]$ there are at most two pairs of the
form $(r,k)$, being $(r,\pi_{r}(e))$ and $(r,\pi_{r}(f))$. Since
$g$ is a false positive and $\pi_{r}(g)\neq\pi_{r}(e)$, we conclude
that $\pi_{r}(g)=\pi_{r}(f)$. 

Using the same argument, starting from $g\neq f$ we conclude that
there is $s$ such that $\pi_{s}(g)\neq\pi_{s}(f)$ and $\pi_{s}(g)=\pi_{s}(e)$.

Obviously $r\neq s$; assume that $r<s$. For these values of $r$
and $s$, in $[P]$ there are at most two triples of the form $(r,s,k)$,
being $(r,s,\pi_{r}(e)+\pi_{s}(e))$ and $(r,s,\pi_{r}(f)+\pi_{s}(f))$.
Since $g$ is a false positive, $\pi_{r}(g)+\pi_{s}(g)$ is equal
to either $\pi_{r}(e)+\pi_{s}(e)$ or $\pi_{r}(f)+\pi_{s}(f)$. If
$\pi_{r}(g)+\pi_{s}(g)=\pi_{r}(e)+\pi_{s}(e)$, since $\pi_{s}(g)=\pi_{s}(e)$,
we conclude that $\pi_{r}(g)=\pi_{r}(e)$, and this fact contradicts
our earlier conclusion $\pi_{r}(g)\neq\pi_{r}(e)$. Likewise, the
case $\pi_{r}(g)+\pi_{s}(g)=\pi_{r}(f)+\pi_{s}(f)$ is also impossible.
Thus, there are no false positives.
\end{proof}
For $R=1$ the star labelling coincides with the bit-per-edge encoding,
and this value of $R$ is best to use when $n$ is small. For larger
values of $n$, larger values of $R$ become optimal; for each specific
size of a star, the optimal value of $R$ can be found simply by calculating
the size $|U|=\left(R+\frac{R(R-1)}{2}\right)\left\lceil \sqrt[R]{n}\:\right\rceil $
for all reasonably possible values of $R$, that is, for $R$ ranging
from $1$ to $\log_{2}n$ (because, obviously, $2$ is the smallest
possible value of $K$). The following table shows optimal values
of $R$ and the corresponding size $|U|$ for some sizes of stars.
As you can see, even for unrealistically large values of $n$ the
size $|U|$ remains reasonably small. Our computational experiments\footnote{We are grateful to an anonymous referee who has suggested the following sketch of a proof as to why this growth rate is observed. The size of $U$ is $O(R^2 n^{1/R})$, so one can take the derivative of the function $f(R)=R^2 n^{1/R}$ and set it equal to zero to determine the local optimum, which turns out to be at $R=O(\log n)$, giving a size of $U$ of $O(\log^2 n)$.}
show that if the optimal value of $R$ is used, the size $|U|$ of
the star labelling grows at the rate $O(\log^{2}n)$. For comparison,
the column `theoretical smallest size' is calculated as the logarithm
to the base $2$ of the total number of shortest paths in the graph;
that is, this is the smallest number of bits needed to distinguish
between shortest paths. The theoretical smallest size grows at the
rate $O(\log n)$.

\begin{tabular}{|c|c|c|c|}
\hline 
$n=|E|$ & theoretical smallest size & $|U|$ & optimal $R$\tabularnewline
\hline 
\hline 
$10$ & $6$ & $10$ & $1$\tabularnewline
\hline 
$10^{2}$ & $13$ & $30$ & $2$\tabularnewline
\hline 
$10^{3}$ & $19$ & $60$ & $3$\tabularnewline
\hline 
$10^{4}$ & $26$ & $100$ & $4$\tabularnewline
\hline 
$10^{5}$ & $33$ & $147$ & $6$\tabularnewline
\hline 
$10^{6}$ & $39$ & $210$ & $6$\tabularnewline
\hline 
\end{tabular}

\section{Decomposing graphs for encoding\label{sec:Decomposing-graphs-for}}

Suppose a graph $G$ can be decomposed into
its core $C$ and its periphery $P$; now we shall describe how to combine labellings for $C$ and $P$ into a labelling for $G$. Consider a graph $G=(V,E)$
and a subset $V_{C}\subset V$. Let $C=(V_{C},E_{C})$ be the subgraph
induced by $V_{C}$, and let $P=(V_{P},E_{P})$ be the graph produced
from $G$ by contracting $C$ into one vertex. 

Suppose $G$ is a graph decomposed into a core $C$ and a periphery
$P$. Suppose labellings are introduced in $C$ and $P$. Then each edge $e\in E_{C}$ has a label $[e]_{C}\in U$
and each edge $e\in E_{P}$ has a label $[e]_{P}\in V$, where $U$
and $V$ are two sets; assume that $U$ and $V$ are disjoint. Slightly
abusing notation, we shall identify edges in $E_{P}$ with the corresponding
edges in $E$. Let the universal set for labelling edges in $G$ be
$W=U\cup V$, and for each $e\in E$ let $[e]=[e]_{C}$ if $e\in E_{C}$
and $[e]=[e]_{P}$ if $e\in E_{P}$. Let us refer to this labelling
as the \emph{combined labelling}. Note that the size $|W|$ is the
sum of sizes $|U|+|V|$.
\begin{proposition}
Assume that the periphery $P$ is a tree. Assume that the labellings used for $C$ and $P$ produce no false positives when used to encode shortest paths. Then if the combined labelling is used to represent a shortest path then it has no false positives.\end{proposition}
\begin{proof}
For each shortest path $S$ in $G$, denote by $S_{C}$ the set of
the edges of $S$ contained in $C$, and denote by $S_{P}$ the set
of the edges of $S$ contained in $P$. It is obvious that both $S_{C}$
and $S_{P}$ are not just sets of edges, but paths, and, moreover,
shortest paths. Note that $[S]=[S_{C}]\cup[S_{P}]$ and $[S_{C}]=[S_{C}]_{C}\subseteq U$
and $[S_{P}]=[S_{P}]_{P}\subseteq V$. Assume that there is a false
positive $f\in E$; that is, $f\notin S$ and $[f]\in[S]$. Hence,
we conclude that $[f]\in[S_{C}]_{C}$ or $[f]\in[S_{P}]_{P}$. From
$[f]\in[S_{C}]_{C}$ it follows that $[f]\subseteq U$ and, therefore,
$f\in E_{C}$ and $[f]=[f]_{C}$. By assumption, labelling $[\cdot]_{C}$
has no false positives; therefore, it is not possible to have $f$
such that $f\notin S_{C}$ and $[f]_{C}\in[S_{C}]_{C}$. Similarly,
the case $[f]\in[S_{P}]_{P}$ is also impossible.
\end{proof}

\includegraphics[scale=0.2]{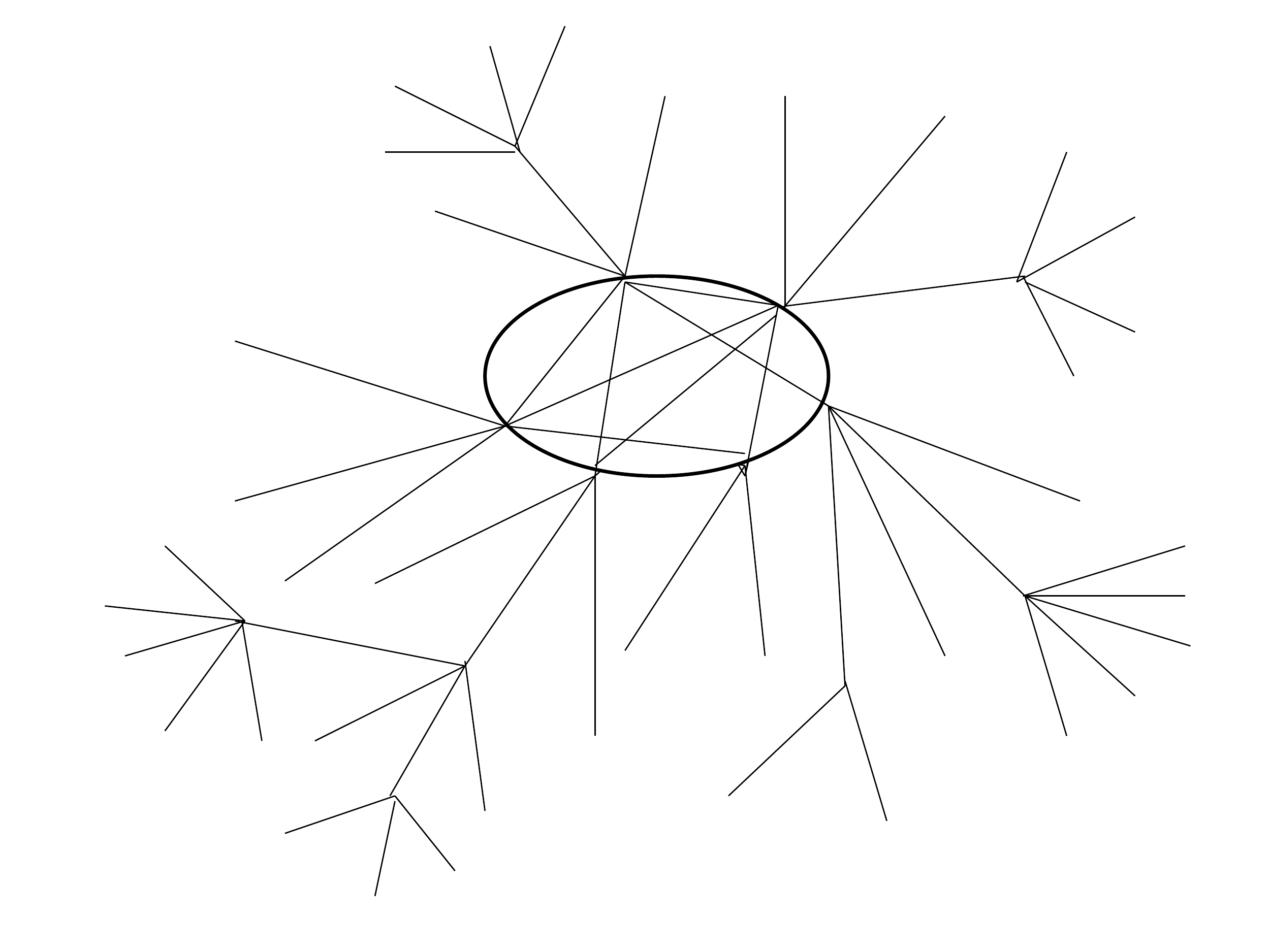}~~~~~~~~~~~~\includegraphics[scale=0.2]{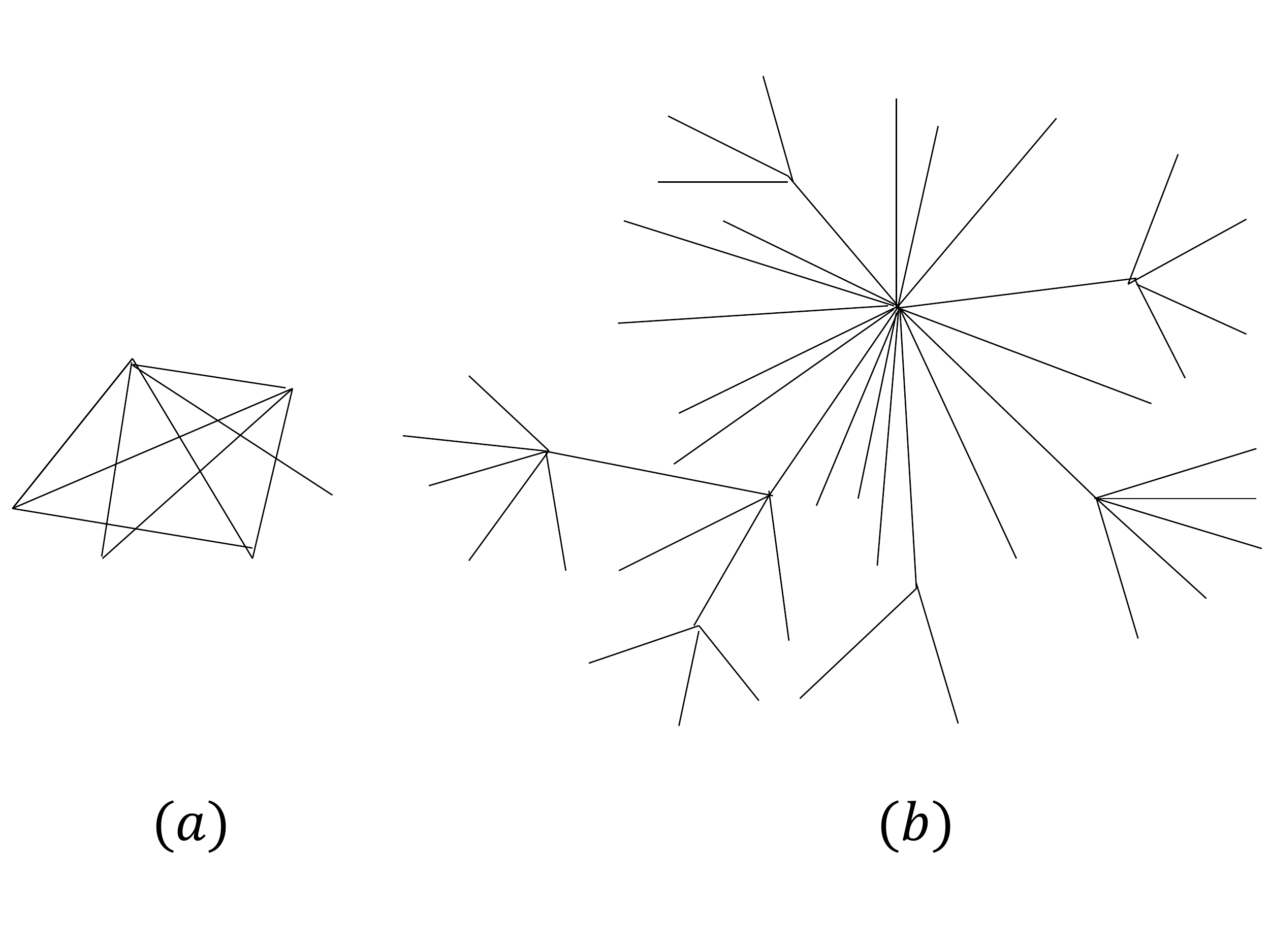}~~~~~~~~~~~~

\includegraphics[scale=0.2]{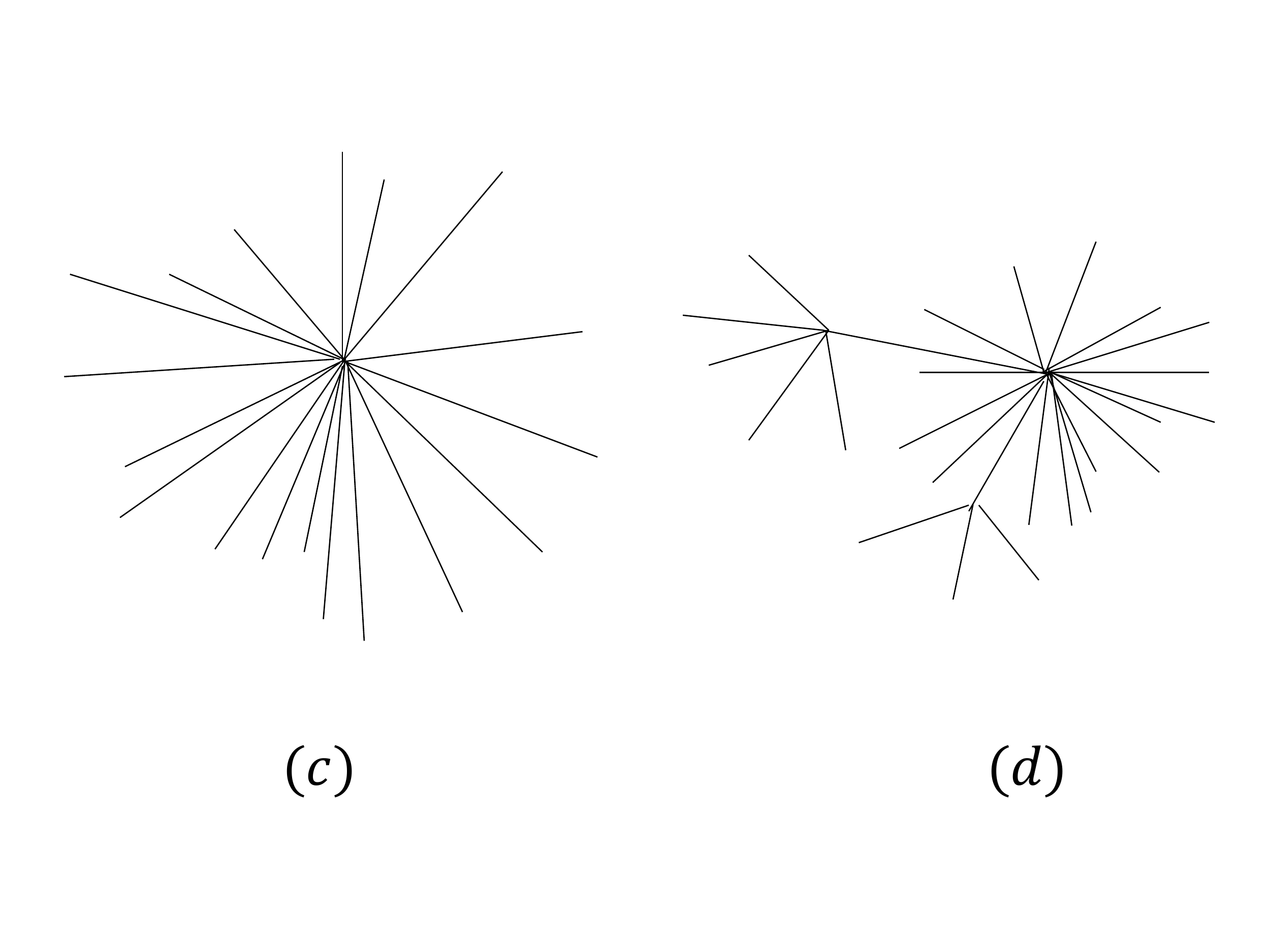}~~~~~~~~~~~~\includegraphics[scale=0.2]{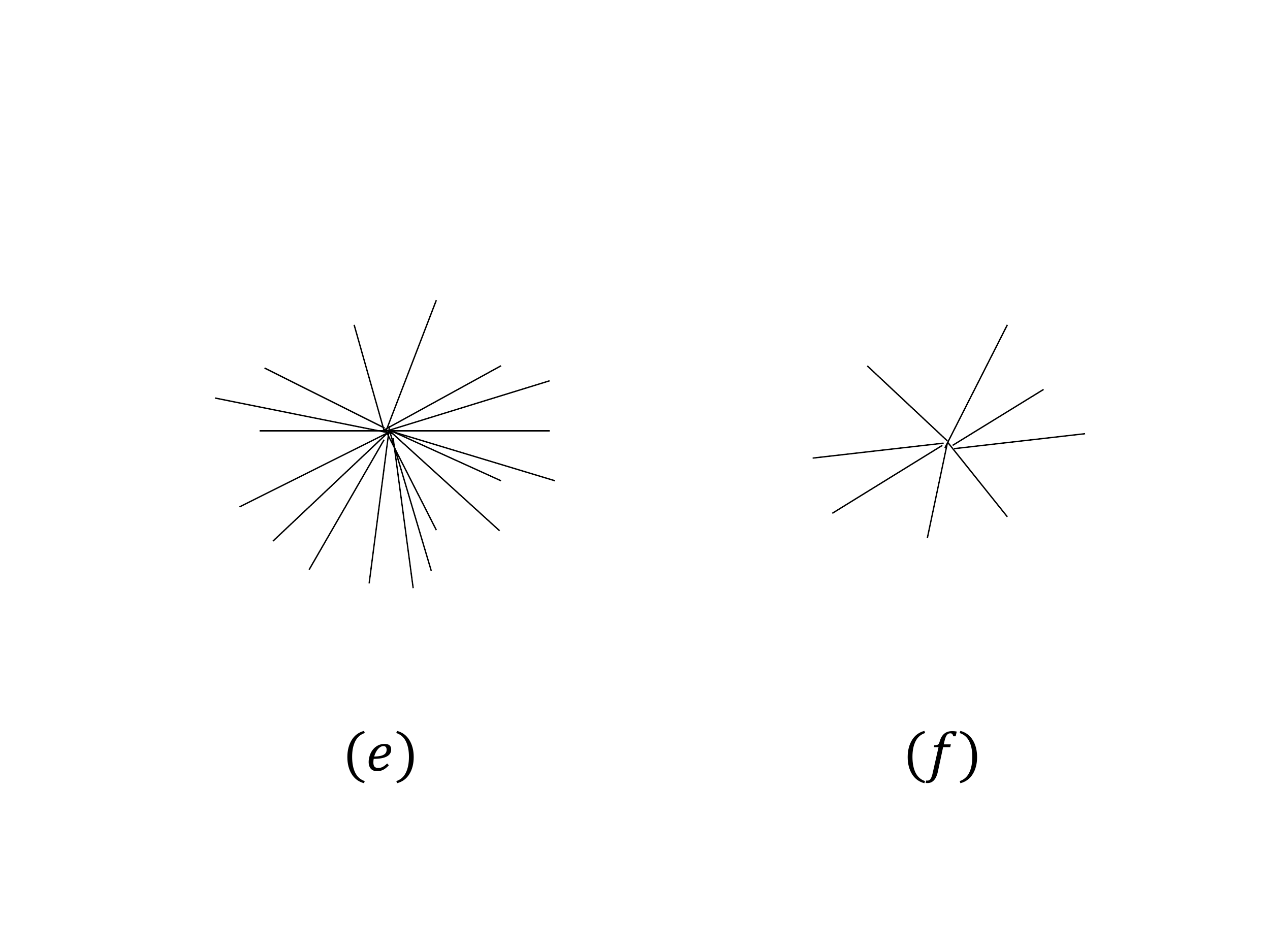}

Let us consider an example of a useful decomposition of a graph. The
graph on the figure, with a dense core (circled) and a tree-like periphery
can be decomposed into the core (a) and the periphery (b), as shown; then the bit-per-vertex labelling for its core can be introduced
as in Subsection \ref{sub:Bit-per-vertex}. To deal with the tree (b),
decompose it into the central star (c) and the periphery of the tree (d),
as shown; continue in the same fashion until the periphery of the tree is completely decomposed into stars, like (d) in our example is decomposed into stars (e), (f).
Then for each star, star labelling can be introduced as in Subsection \ref{sub:Encoding-for-star}. After that, combine all these labellings using the combined labelling described in this subsection.

\section{Models of computer networks}

In this section we consider several kinds of examples inspired by
practice.

\subsection{Core-periphery graphs}

Suppose $n$ is a positive integer. Suppose a graph $G$ consists
of a dense graph with $n$ vertices (this is the core of $G$), and
each of these $n$ vertices is adjacent to $n-1$ hanging vertices
(this is the periphery of $G$). Thus, there are $n^{2}$ vertices
in $G$ in total, and $G$ decomposes into a dense graph with $n$
vertices and a star with $n(n-1)$ edges. Suppose we label edges in
the core using the bit-per-vertex labelling, as in Subsection \ref{sub:Bit-per-vertex}, and edges in the
periphery using the star labelling, as in Subsection \ref{sub:Encoding-for-star}, and then
combine these labels using the combined labelling, as in Section \ref{sec:Decomposing-graphs-for}.
The table below shows the combined size $|U|$ needed for routing
in $G$.

\begin{tabular}{|c|c|c|c|c|}
\hline 
$n$ & $|V|=n^{2}$ & $|E|$ & theoretical smallest size & $|U|$\tabularnewline
\hline 
\hline 
$100$ & $10000$ & $14850$ & $26$ & $200$\tabularnewline
\hline 
$200$ & $40000$ & $59700$ & $30$ & $326$\tabularnewline
\hline 
$300$ & $90000$ & $134550$ & $32$ & $447$\tabularnewline
\hline 
$400$ & $160000$ & $239400$ & $34$ & $565$\tabularnewline
\hline 
$500$ & $250000$ & $374250$ & $35$ & $668$\tabularnewline
\hline 
\end{tabular}

\subsection{Binary trees}

Consider a rooted binary tree $G$ of height $h$; more precisely,
we require that the tree be \emph{perfect}, that is, every vertex
except leaves has exactly two children, and every leaf is at the distance
$h$ from the root. Suppose we decompose $G$ into stars, following
the procedure in Section \ref{sec:Decomposing-graphs-for}, by first
contracting the $2$-edge star centered at the root of the tree, then contracting the
$4$-edge star centered at the root, etc., until the last remaining
tree is a $2^{h}$-edge star (corresponding to the hanging edges of
the original tree). Suppose we label edges in each star using the star labelling, as in Subsection
\ref{sub:Encoding-for-star}, and then combine these labels using the combined labelling, as in
Section \ref{sec:Decomposing-graphs-for}. The table below shows the
combined size $|U|$ needed for routing in $G$. Our computational
experiments show that the size $|U|$ grows at the rate $O(h^{3})$.
For comparison, the theoretical smallest size grows at the rate $O(h)$.

\begin{tabular}{|c|c|c|c|}
\hline 
$h$ & $|V|$ & theoretical smallest size & $|U|$\tabularnewline
\hline 
\hline 
$5$ & $63$ & $8$ & $44$\tabularnewline
\hline 
$10$ & $2047$ & $15$ & $252$\tabularnewline
\hline 
$15$ & $65535$ & $22$ & $733$\tabularnewline
\hline 
\end{tabular}

\subsection{Random labels} \label{sub:Bloom}

In the research that preceded ours \cite{Jokela2009,carrea2014optimized}, labels $[e]$ for edges are chosen not in the way we do it in this paper, but as random subsets of $U$ of a fixed size $k$; such labels are called Bloom filters. Bloom filters were introduced in \cite{bloom1970space} and then studied by various authors; see, for example, the survey article \cite{broder2004network}. Bloom filters produce false positives with a certain probability, which is usually approximated \cite{bloom1970space} by the formula $\left(1-e^{\frac{-kn}{m}}\right)^{k}$, where $m=|U|$, $n=|S|$ and $k$ is the size of $[e]$ for each $e\in E$.  

The table below demonstrates the false positive rates for stars with a varying size $|E|$ if random Bloom filters are used instead of star labelling to encode paths of length $2$. In this comparison, the parameters of random Bloom filters are the same as with star labelling, that is, $m=|U|$ and $n=2$; however, instead of using $k=R+\frac{R(R-1)}{2}$, we optimise it to make this comparison fairer towards Bloom filters. Indeed \cite{broder2004network}, if the expected size $n$ is known in advance, the value of the parameter $k$ can be chosen to minimise the probability of false positives when representing sets of size $n$; thus, we use the optimal value, which is given by the formula $k = \frac{m}{n}\ln 2$. 

Here is how the numbers in the table should be interpreted. We assume that each message is delivered from a leaf of the star to another leaf; that is, the delivery path consists of $2$ edges. For $|E|=40$ the false positive rate is $0.6\%$; this means that any given edge out of the $38$ edges which are not on the path can be a false positive with probability $0.6\%$; thus, there is a total probability $20\%$ that there will be at least one false positive. Therefore, with roughly every fifth message, the central node of the star will not be able to find out along which edge the message should be forwarded. 

\begin{tabular}{|c|c|c|}
\hline 
$|E|$ & $|U|$ & false positive rate\tabularnewline
\hline 
\hline 
$10$ & $10$ & $9.1\%$\tabularnewline
\hline 
$20$ & $15$ & $2.7\%$\tabularnewline
\hline 
$30$ & $18$ & $1.3\%$\tabularnewline
\hline 
$40$ & $21$ & $0.6\%$\tabularnewline
\hline 
\end{tabular}

\section{Conclusion}

We started this work after several colleagues conducting research in electronic engineering told us that in their experiments, using random Bloom filters (as described in Subsection \ref{sub:Bloom}) seems a reasonably good approach, producing relatively few false positives when a reasonably small size $|U|$ is used. Our results convincingly show that with some light assumptions (we assume that the network has a certain shape, and that messages are delivered along shortest paths) it is possible to introduce labelling which both requires a reasonably small size $|U|$ and has no false positives at all.

\bibliographystyle{siamplain}
\bibliography{references-tree-dense}
\end{document}